\documentclass[oneside,english]{amsart}
\usepackage[T1]{fontenc}
\usepackage[latin9]{inputenc}
\usepackage{amstext}
\usepackage{amsthm}
\usepackage{amssymb}

\makeatletter
\numberwithin{equation}{section}
\numberwithin{figure}{section}
\theoremstyle{plain}
\newtheorem{thm}{\protect\theoremname}
\theoremstyle{plain}
\newtheorem{lem}[thm]{\protect\lemmaname}
\theoremstyle{remark}
\newtheorem{rem}[thm]{\protect\remarkname}
\theoremstyle{corollary}
\newtheorem{cor}[thm]{\protect\corollaryname}

\numberwithin{thm}{section}
\usepackage{xcolor}

\makeatother

\usepackage{babel}
\providecommand{\lemmaname}{Lemma}
\providecommand{\remarkname}{Remark}
\providecommand{\theoremname}{Theorem}
\providecommand{\corollaryname}{Corollary}

\begin{document}
\title[Superconvergence for arbitrary triangular arrays]{Superconvergence in free probability limit theorems for arbitrary triangular arrays}
\author{Hari Bercovici, Ching-Wei Ho, Jiun-Chau Wang, and Ping Zhong}
\begin{abstract}
It is known that limit theorems for triangular arrays with identically distributed rows yields convergence of densities rather than just convergence in distribution. We show that this superconvergence result holds---at least at points at which the limit density is nonzero---even if the rows of the array are not identically distributed.
\end{abstract}

\date{\today}
\subjclass[2010]{46L54}
\keywords{free convolution, superconvergence}
\address{Hari Bercovici: Department of Mathematics, Indiana University, Bloomington,
IN 47405, United States }
\email{bercovic@indiana.edu}
\address{Ching-Wei Ho: Institute of Mathematics, Academia Sinica, Taipei 10617, Taiwan; Department of Mathematics, University of Notre Dame, Notre Dame,
IN 46556, United States}
\email{cho2@nd.edu}
\address{Jiun-Chau Wang: Department of Mathematics and Statistics, University
of Saskatchewan, Saskatoon, S7N 5E6, Canada }
\email{jcwang@math.usask.ca}
\address{Ping Zhong: Department of Mathematics and Statistics, University of
Wyoming, Laramie, WY 82071, United States}
\email{pzhong@uwyo.edu}

\maketitle

\section{Introduction \label{sec:Introduction}}

The central limit theorem of free probability \textbf{\cite{V-1983}}
asserts that for a sequence of freely independent,
identically distributed random variables $X_{1},X_{2},\dots$ with zero expected value
and unit variance, the variables 
\[
Y_{n}=\frac{X_{1}+\cdots+X_{n}}{\sqrt{n}}
\]
converge in distribution to the standard semicircular law with density
\[
\frac{1}{2\pi}\sqrt{4-t^{2}}
\]
for $t\in[-2,2]$. In case the variables $X_{n}$ are also bounded,
it was observed in \textbf{\cite{BVsuper}} that the distribution
of $Y_{n}$ is necessarily absolutely continuous for large $n$, and
that the densities of these distributions (along with their derivatives
of all orders) converge uniformly to the semicircular density on every
interval $[a,b]\subset(-2,2)$. Moreover, the distribution of $Y_{n}$
is supported on an interval $[c_{n},d_{n}]$ such that $\lim_{n\to\infty}c_{n}=-2$
and $\lim_{n\to\infty}d_{n}=2$. This phenomenon was called \emph{superconvergence}
in \textbf{\cite{BVsuper}}.  Chistyakov and G\"otze computed an asymptotic expansion of the free central limit theorem \cite{CG-PTRF} and the rate of convergence of the density \cite{CG-arXiv}. Various aspects of superconvergence were extended
to more general limit processes for sums as well as for products of
\emph{infinitesimal arrays} of random variables, generally under the
assumption that the rows of the arrays are formed of freely independent,
identically distributed variables (see \cite{AWZ,BWZ-Pacific,BWZ-mul, K1,K2}).
The purpose of this note is to remove the identical distribution hypothesis
on the arrays. The technique is closer to the one employed in \textbf{\cite{BVsuper}}
than to the more recent developments and it essentially uses only
the fact that the $R$-transforms of the variables in an infinitesimal
array are defined on arbitrarily large Stolz angles. In particular,
the subordination property of free convolution is not needed in this
paper. The results do not even require that we work with a convolution
of infinitesimal variables, but they only apply locally to the points
that have a neighborhood where the limit distribution is absolutely
continuous. We provide complete arguments for the additive case. The
arguments for products are very similar and only the relevant differences
are pointed out.

\section{Additive convolution and superconvergence\label{sec:Additive-convolution-and}}

The distribution of a sum of freely independent selfadjoint random
variables is the \emph{free additive convolution} of the distributions
of the summands. Because of this fact, we can forget about the variables
themselves and focus instead on free convolutions. Suppose that
$\mu$ is a Borel probability measure on $\mathbb{R}$, and define
two analytic functions on $\mathbb{H}=\{z\in\mathbb{C}:\Im z>0\}$
by
\[
G_{\mu}(z)=\int_{\mathbb{R}}\frac{d\mu(t)}{z-t},\quad F_{\mu}(z)=\frac{1}{G_{\mu}(z)},\quad z\in\mathbb{H}.
\]
The measure $\mu$ is uniquely determined by either of these functions,
and the Stieltjes inversion formula shows that the value of the density
of $\mu$ at almost every $x\in\mathbb{R}$ is equal to the boundary
limit $-\pi^{-1}\lim_{y\downarrow0}\Im G_{\mu}(x+iy)$. The function
$F_{\mu}$ satisfies
\begin{equation}
\lim_{z\to\infty}\frac{F_{\mu}(z)}{z}=1,\label{eq:F(z)/z}
\end{equation}
where the limit is taken as $z=x+iy\to\infty$ such that $|x|/y$
remains bounded. Given $\alpha,\beta>0$, define the Stolz angle at infinity by 
\[
\Gamma_{\alpha,\beta}=\{x+iy\in\mathbb{C}:y\ge\max\{\beta,\alpha|x|\}\}.
\]
Then (\ref{eq:F(z)/z}) implies that for every $\alpha>0$, there exist
$\beta>0$ and an analytic function $H_{\mu}:\Gamma_{\alpha,\beta}\to\mathbb{H}$
such that 
\[
F_{\mu}(H_{\mu}(z))=z,\quad z\in\Gamma_{\alpha,\beta};
\]
see \cite{BVunbdd}. Since $\Im F_{\mu}(z)\ge\Im z$ for $z\in\mathbb{H},$
it follows that the function
\[
\varphi_{\mu}(z)=H_{\mu}(z)-z
\]
satisfies
\begin{equation}
\Im\varphi_{\mu}(z)\le0,\quad z\in\Gamma_{\alpha,\beta}.\label{eq:phi has negative im}
\end{equation}
The function $\varphi_{\mu}$ may have analytic continuations to other
regions of the form $\Gamma_{\alpha',\beta'}$ that still satisfy
(\ref{eq:phi has negative im}) but possibly $\Im(z+\varphi_{\mu}(z))\le0$
for some $z\in\Gamma_{\alpha',\beta'}$, so the quantity $F_{\mu}(z+\varphi_{\mu}(z))$
is not defined. Any two such analytic continuations coincide on their
common domain of definition, and we continue denoting by $\varphi_{\mu}$
the function obtained by assembling all of these possible analytic
continuations to domains of the form $\Gamma_{\alpha',\beta'}$. Thus,
if $\varphi_{\mu}$ is defined at a point $z=x+iy\in\mathbb{H}$,
then it is also defined at all points $x'+iy'$ such that $y'\ge y$
and $|x'|/y'\le|x|/y$. The equation
\[
F_{\mu}(z+\varphi_{\mu}(z))=z
\]
persists for $z$ in every connected open set $U$ such that $U\subset\Gamma_{\alpha',\beta'}$
for some $\alpha',\beta'>0$ and $z+\varphi_{\mu}(z)\in\mathbb{H}$
for every $z\in U$. The related identity
\begin{equation}
F_{\mu}(z)+\varphi_{\mu}(F_{\mu}(z))=z\label{eq:H(F)=00003Dz}
\end{equation}
holds in every connected open set $V\subset\mathbb{H}$ with the property
that $F_{\mu}(z)$ is in the domain of $\varphi_{\mu}$ for $z\in V$.

The function $\varphi_{\mu}$ defined above is known as the \emph{Voiculescu
transform} of $\mu$ and it serves to linearize the free additive
convolution of probability measures on $\mathbb{R}$; that is
\[
\varphi_{\mu_{1}\boxplus\mu_{2}}=\varphi_{\mu_{1}}+\varphi_{\mu_{2}}
\]
on the common domain of $\varphi_{\mu_{1}}$ and $\varphi_{\mu_{2}}$.
We also recall that a measure $\mu$ is $\boxplus$-\emph{infinitely
divisible} if and only if $\varphi_{\mu}$ is defined on the entire
$\mathbb{H}$ and $\varphi_{\mu}(\mathbb{H})\subset\mathbb{C}\backslash\mathbb{H}$
\cite{BVunbdd}. If $\mu$ is $\boxplus$-infinitely divisible, the
identity (\ref{eq:H(F)=00003Dz}) holds for for every $z\in\mathbb{H}$
and thus $F_{\mu}$ maps $\mathbb{H}$ conformally onto a domain $\Omega_{\mu}\subset\mathbb{H}.$
It was shown in \cite{BVunbdd} that 
\[
\Omega_{\mu}=\{z\in\mathbb{H}:z+\varphi_{\mu}(z)\in\mathbb{H}\}.
\]
In fact, there exists a continuous function $f_{\mu}:\mathbb{R}\to[0,+\infty)$
such that
\[
\Omega_{\mu}=\{x+iy:x\in\mathbb{R},y>f_{\mu}(x)\},
\]
and the map $x\mapsto H_{\mu}(x+if_{\mu}(x))$ (extended by continuity
when $f_{\mu}(x)=0$) is a homeomorphism of $\mathbb{R}$ onto $\mathbb{R}$,
see\textbf{ \cite{Huang}} or \cite[Section 2]{BWZ-Pacific}. Moreover,
the function $f_{\mu}$ is real-analytic wherever it is nonzero. If
$f_{\mu}>0$ on an interval $[A,B]$ then $\mu$ has a real-analytic
density on the interval $[a,b]$, where 
\[
a=H_{\mu}(A+if_{\mu}(A)),\quad b=H_{\mu}(B+if_{\mu}(B)),
\]
and the density at a point $t=H_{\mu}(s+if_{\mu}(s))$, $s\in[A,B]$,
is given by
\[
-\frac{1}{\pi}\Im G_{\mu}(t)=-\frac{1}{\pi}\Im\frac{1}{s+if_{\mu}(s)}=\frac{1}{\pi}\frac{f_{\mu}(s)}{s^{2}+f_{\mu}(s)^{2}}.
\]
Every interval $[a,b]$ on which $\mu$ has a positive density is
of the form described above. The following fact is used in the proof
of the main result.
\begin{lem}
\label{lem:y-derivative>0} Suppose that $\mu$ is a $\boxplus$-infinitely
divisible measure on $\mathbb{R},$ and let $s\in\mathbb{R}$ be such
that $f_{\mu}(s)>0$. Then
\[
\frac{\partial}{\partial t}\Im H_{\mu}(s+it)>0\text{ for }t\ge f_{\mu}(s).
\]
\end{lem}

\begin{proof} When $\mu$ is a point mass, we have
 $	\Im H_{\mu}(s+it)=t$, so the result is immediate.  Suppose therefore that $\mu$ is not a point mass and write the function $\varphi_{\mu}$ in its Nevanlina representation
\[
\varphi_{\mu}(z)=c+\int_{\mathbb{R}}\frac{1+zx}{z-x}\,d\sigma(x),
\]
where $c$ is a real constant and $\sigma$ is a nonzero, finite, positive
Borel measure on $\mathbb{R}$.  By the Cauchy-Riemann equations, it suffices
to show that
\[
\Re H'_{\mu}(z)>0\text{ for }z=s+it,\,t\ge f_{\mu}(s).
\]
Easy calculations show that
\[
H'_{\mu}(s+it)=1-\int_{\mathbb{R}}\frac{1+x^{2}}{(s+it-x)^{2}}d\sigma(x),
\]
and
\[
\Im H_{\mu}(s+it)=t\left[1-\int_{\mathbb{R}}\frac{1+x^{2}}{|s+it-x|^{2}}d\sigma(x)\right].
\]
The fact that $f_{\mu}(s)>0$ means that
\[
\int_{\mathbb{R}}\frac{1+x^{2}}{|s+it-x|^{2}}d\sigma(x)=1,\quad t=f_{\mu}(s),
\]
and thus
\[
\int_{\mathbb{R}}\frac{1+x^{2}}{|s+it-x|^{2}}d\sigma(x)\le1,\quad t\ge f_{\mu}(s).
\]
 On the other hand, one has 
\[
\Re\int_{\mathbb{R}}\frac{1+x^{2}}{(s+it-x)^{2}}d\sigma(x)<\left|\int_{\mathbb{R}}\frac{1+x^{2}}{(s+it-x)^{2}}d\sigma(x)\right|\le\int_{\mathbb{R}}\frac{1+x^{2}}{|s+it-x|^{2}}d\sigma(x)\le1
\]
for $t\ge f_{\mu}(s)$, and this implies the desired inequality $\Re H'_{\mu}(z)>0$.
\end{proof}
\begin{rem}
\label{rem:y-derivative again}In the above proof, we were able to
use an explicit formula for the derivative. Calculations are more
cumbersome for free multiplicative convolutions and therefore the
following general fact will be useful. Suppose that $\alpha\in\mathbb{R}$
and $\varepsilon,\delta>0$, and that $f:(\alpha-\varepsilon,\alpha+\varepsilon)\to\mathbb{R}$
is a differentiable function. Set 
\[
D=\{x+iy:x\in(\alpha-\varepsilon,\alpha+\varepsilon)\text{ and }|y-f(x)|<\delta\},
\]
and let $H:D\to\mathbb{C}$ be an analytic function such that
\[
\Im H(x+iy)\begin{cases}
>0, & y>f(x),\\
<0, & y<f(x).
\end{cases}
\]
Then
\begin{equation}
\frac{\partial}{\partial y}\Im H(x+iy)>0\text{ for }y=f(x),\:x\in(\alpha-\varepsilon,\alpha+\varepsilon).\label{eq:42}
\end{equation}
To see this, we observe that the function $H(x+if(x))$ is necessarily
real and increasing on $(\alpha-\varepsilon,\alpha+\varepsilon)$,
and thus
\[
0\le\frac{d}{dx}H(x+if(x))=H'(x+if(x))(1+if'(x)).
\]
 Writing $H'(x+if(x))=a+ib$ with $a,b\in\mathbb{R}$, we see that
$b=-af'(x)$ and
\[
\frac{d}{dx}H(x+if(x))=a-bf'(x)=a(1+f'(x)^{2}).
\]
Since it is easily seen that $H'(x+if(x))\ne0$ (otherwise, the function
$H(x+if(x))$ behaves locally as a power function, and $H(x+if(x))$
cannot be real on $(\alpha-\varepsilon,\alpha+\varepsilon)$), we
conclude that $a>0$, and (\ref{eq:42}) follows from the Cauchy--Riemann
equations.
\end{rem}

Weak convergence is easily described in terms of the functions $F_{\mu}$
or $\varphi_{\mu}$. If a sequence $\{\mu_{n}\}_{n\in\mathbb{N}}$
of Borel probability measures is tight, then there exist $\alpha,\beta>0$
such that $\Gamma_{\alpha,\beta}$ is contained in the domain of $\varphi_{\mu_{n}}$
for every $n\in\mathbb{N}$. Moreover, the weak convergence of $\{\mu_{n}\}_{n\in\mathbb{N}}$
to $\mu$ is equivalent to the local uniform convergence of $\varphi_{\mu_{n}}$
to $\varphi_{\mu}$ on $\Gamma_{\alpha,\beta}$, as well as to the
local uniform convergence of $F_{\mu_{n}}$ to $F_{\mu}$ on $\mathbb{H}$.
See \textbf{\cite{BVunbdd}} for the proofs of these results.

Suppose now that $k_{1},k_{2},\dots\in\mathbb{N}$, and that $\{\mu_{n,i}:n,i\in\mathbb{N},i\le k_{n}\}$
is an array of Borel probability measures on $\mathbb{R}$. This array
is said to be \emph{infinitesimal} if 
\[
\lim_{n\to\infty}\min_{1\le i\le k_{n}}\mu_{n,i}\left((-\varepsilon,\varepsilon)\right)=1
\]
for every $\varepsilon>0$. It was observed in \textbf{\cite{BP-PAMS}
}that, given an infinitesimal array as above and arbitrary $\alpha,\beta>0$,
there exists $N\in\mathbb{N}$ such that $\varphi_{\mu_{n,i}}$ is
defined on $\Gamma_{\alpha,\beta}$ for every $n\ge N$ and $i=1,\dots,k_{n}$.
In particular, if we set
\[
\nu_{n}=\mu_{n,1}\boxplus\cdots\boxplus\mu_{n,k_{n}},
\]
$\varphi_{\nu_{n}}$ is defined on $\Gamma_{\alpha,\beta}$ for $n\ge N$.
Therefore the following result applies in particular to free additive
convolutions of measures in an infinitesimal array and provides our
extension of superconvergence to such an array.
\begin{thm}
\label{thm:general additive} Let $\{\nu_{n}\}_{n\in\mathbb{N}}$
be a sequence of Borel probability measures on $\mathbb{R}$ that
converges weakly to an $\boxplus$-infinitely divisible measure $\nu.$
Suppose that for every $\alpha,\beta>0$ there exists $N\in\mathbb{N}$
such that $\varphi_{\nu_{n}}$ is defined in $\Gamma_{\alpha,\beta}$
for every $n\ge N$. Let $J\subset\mathbb{R}$ be a compact interval
such that $\nu$ is absolutely continuous and $d\nu/dx>0$ in a neighborhood
of $J$. Then $\nu_{n}$ is absolutely continuous on a neighborhood
of $J$ with a real-analytic density for sufficiently large $n$,
and the densities $d\nu_{n}/dx$ converge uniformly on $J$, along
with all their derivatives, to $d\nu/dx$ as $n\to\infty$.
\end{thm}

\begin{proof}
Since unit point masses are purely singular we may, and shall, assume
that the support of $\nu$ contains more than one point. It suffices
to show that $d\nu_{n}/dx$ converges locally uniformly to $d\nu/dx$
on a set on which the latter density is strictly positive. Fix $\alpha_{0},\beta_{0}>0$
such that $\Gamma_{\alpha_{0},\beta_{0}}$ is contained in the domain
of $\varphi_{\nu_{n}}$ for every $n\in\mathbb{N}$ and $\varphi_{\nu_{n}}$
converges locally uniformly to $\varphi_{\nu}$ on $\Gamma_{\alpha_{0},\beta_{0}}$.
As noted above, the function $\varphi_{\nu}$ is defined on $\mathbb{H}$.
Fix a point $x_{0}\in\mathbb{R}$ such that $\nu$ is absolutely continuous
in a neighborhood of $x_{0}$ and $(d\nu/dx)(x_{0})>0$. It follows
that there exists $s_{0}\in\mathbb{R}$ such that $f_{\nu}(s_{0})>0$
and $x_{0}=H_{\nu}(s_{0}+if_{\nu}(s_{0}))$. Given an arbitrary number
$\varepsilon>0$ with $\varepsilon<f_{\nu}(s_{0})$, Lemma \ref{lem:y-derivative>0}
allows us to choose an interval $[A,B]$ containing $s_{0}$ such
that 
\[
Q=\{s+it:s\in[A,B],|t-f_{\nu}(s_{0})|\le\varepsilon\}
\]
 is contained in $\mathbb{H}$, $|f_{\nu}(s)-f_{\nu}(s_{0})|<\varepsilon$
for $s\in[A,B]$, and
\[
\frac{\partial}{\partial t}\Im H_{\nu}(s+it)>0,\quad s+it\in Q.
\]
Choose also $\delta>0$ such that the interior of the rectangle
\[
K=\left\{ s+it:s\in[A,B],\delta\le t\le\frac{1}{\delta}\right\} 
\]
 has nonempty intersection with $\Gamma_{\alpha_{0},\beta_{0}}$ and
$Q\subset K$. Finally, choose $\alpha,\beta>0$ such that 
\[
K\cup\Gamma_{\alpha_{0},\beta_{0}}\subset\Gamma_{\alpha,\beta}.
\]
By hypothesis, there exists $N\in\mathbb{N}$ be such that $\varphi_{\nu_{n}}$
is defined on $\Gamma_{\alpha,\beta}$ for every $n\ge N$. The sequence
of restrictions $\{\varphi_{\nu_{n}}|_{\Gamma_{\alpha_{0},\beta_{0}}}\}_{n\ge N}$
converges to $\varphi_{\nu}$ uniformly in some closed disk contained
in $\Gamma_{\alpha_{0},\beta_{0}}$, and since these functions take
values in $\mathbb{-H}$, the Vitali-Montel theorem implies that $\varphi_{\nu_{n}}$
actually converges locally uniformly to $\varphi_{\nu}$ on $\Gamma_{\alpha,\beta}$,
and thus $H_{\nu_{n}}$ also converges locally uniformly to $H_{\nu}$
on $\Gamma_{\alpha,\beta}$. In particular, $H_{\nu_{n}}$ converges
uniformly to $H_{\nu}$ on a neighborhood of $K$ so, after replacing
$N$ by a larger value, we may assume that
\begin{equation}
\frac{\partial}{\partial t}\Im H_{\nu_{n}}(s+it)>0,\quad s+it\in Q,\,n\ge N.\label{eq:imaginary partial}
\end{equation}
Note also that the set
\[
K_{+}=\{s+it\in K:t\ge f_{\nu}(s_{0})+\varepsilon\}
\]
is contained in $\Omega_{\nu}$, while
\[
K_{-}=\{s+it\in K:t\le f_{\nu}(s_{0})-\varepsilon\}
\]
is disjoint from $\Omega_{\nu}$. Thus, possibly making $N$ even
larger, we may assume that
\[
\Im H_{\nu_{n}}(z)>0\text{ for }z\in K_{+}\text{ and }\Im H_{\nu_{n}}(z)<0\text{ for }z\in K_{-},
\]
provided that $n\ge N$. Combining this with (\ref{eq:imaginary partial}),
we see that for every $n\ge N$ and for every $s\in[A,B]$ there exists
a unique $f_{n}(s)>0$ such that 
\[
\Im H_{\nu_{n}}(s+it)>0\text{ for }t>f_{n}(s)\text{ and }\Im H_{\nu_{n}}(s+it)<0\text{ for }\delta<t<f_{n}(s).
\]
Of course, we have $|f_{n}(s)-f_{\nu}(s)|<\varepsilon$. Since $\varepsilon$
can be made arbitrarily small, we conclude that $f_{n}$ converges
to $f_{\nu}$ uniformly on $[A,B]$. 

We observe next that $F_{\nu_{n}}(H_{\nu_{n}}(z))=z$ for $z\in\Gamma_{\alpha_{0},\beta_{0}},$
and our choice of $K$, along with analytic continuation, show that
\[
F_{\nu_{n}}(H_{\nu_{n}}(s+it))=s+it,\quad s+it\in Q,\,t>f_{n}(s),\,n\ge N.
\]
Now, the set $Q$ is convex, and thus (\ref{eq:imaginary partial})
implies that $H_{\nu_{n}}$ is injective with an analytic inverse
on a neighborhood of $Q$, and that $F_{\nu_{n}}$ coincides with
the inverse of $H_{\nu_{n}}$ on the set $\{s+it\in Q:t>f_{n}(s)\}$.
It follows that $F_{\nu_{n}}$ has an analytic continuation to a neighborhood
of $H_{\nu_{n}}(Q).$ In particular, $F_{\nu_{n}}$ has an analytic
continuation across the segment
\[
\{H_{\nu_{n}}(s+if_{n}(s)):s\in[A,B]\},
\]
and this segment tends to 
\[
\{H_{\nu}(s+if_{\nu}(s)):s\in[A,B]\}
\]
 which is a neighborhood of $x_{0}$. This shows that $\nu_{n}$ is
absolutely continuous with a real-analytic density $d\nu_{n}/dx$
in a neighborhood of $x_{0}$, and this density satisfies the formula
\[
\frac{d\nu_{n}}{dx}\left(H_{\nu_{n}}(s+if_{n}(s)\right)=\frac{1}{\pi}\frac{f_{n}(s)}{s^{2}+f_{n}(s)^{2}},\quad s\in[A,B].
\]

Finally, to show that the densities $d\nu_{n}/dx$ converge uniformly
to $d\nu/dx$ in a neighborhood of $x_{0}$, it suffices to show that
$F_{\nu_{n}}$ converges uniformly to $F_{\nu}$ in a neighborhood
of $x_{0}$. This, as well as the convergence of the derivatives,
follows from the formula 
\[
F_{\nu_{n}}(z)=H_{\nu_{n}}^{-1}(z)=\frac{1}{2\pi i}\int_{\partial Q}\frac{\zeta H_{\nu_{n}}(\zeta)}{H_{\nu_{n}}(\zeta)-z}\,d\zeta
\]
 that holds for $z\in\mathbb{H}$ close to $x_{0}$ and for sufficiently
large $n$.
\end{proof}

The following corollary formalizes the discussion preceding Theorem~\ref{thm:general additive}.
\begin{cor}
    Suppose that $k_{1},k_{2},\dots\in\mathbb{N}$, and that $\{\mu_{n,i}:n,i\in\mathbb{N},i\le k_{n}\}$
is an infinitesimal array of Borel probability measures on $\mathbb{R}$. Set
\[\nu_n = \mu_{n,1}\boxplus\cdots\boxplus\nu_{n,k_n}.\]
Assume that $\nu_n$ converges to a $\boxplus$-infinite divisible measure $\nu$. Then for any compact interval $J\subset\mathbb{R}$ such that $\nu$ is absolutely continuous 
and $d\nu/dx>0$ in a neighborhood of $J$, $\nu_n$ is absolutely continuous on a neighborhood of $J$ with a real-analytic density for sufficiently large $n$, 
and the densities $d\nu_n/dx$ converge uniformly on $J$ to $d\nu/dx$ as $n\to\infty$.
\end{cor}

\section{Multiplicative convolution on $\mathbb{R}_{+}$\label{sec:Multiplicative-convolution-onR}}

Suppose that $\mu$ is a Borel probability measure on $\mathbb{R}_{+}=[0,+\infty)$,
other than the point mass $\delta_{0}$ at the origin. The analytic
functions
\[
\psi_{\mu}(z)=\int_{\mathbb{R}_{+}}\frac{zt}{1-zt}\,d\mu(t),\quad\eta_{\mu}(z)=\frac{\psi_{\mu}(z)}{1+\psi_{\mu}(z)},\quad z\in\mathbb{C}\backslash\mathbb{R}_{+},
\]
are real-valued on $(-\infty,0)$ and the measure $\mu$ is uniquely
determined by either $\psi_{\mu}$ or $\eta_{\mu}$. Indeed, we have
\begin{equation}
zG_{\mu}(z)=\frac{1}{1-\eta_{\mu}\left(\frac{1}{z}\right)},\quad z\in\mathbb{C}\setminus\mathbb{R}_{+},\label{eq:G and eta}
\end{equation}
to which the Stieltjes inversion can be applied to recover $\mu$.
Moreover, $\psi_{\mu}$ and $\eta_{\mu}$ map $\mathbb{H}$ to itself
and, in addition,
\[
\arg(\eta_{\mu}(z))\ge\arg z,\quad z\in\mathbb{H},
\]
where `$\arg$' stands for the principal value of the argument, so
$\arg z\in(0,\pi)$ for $z\in\mathbb{H}$, see \cite{BB-IMRN}. It
was shown in \textbf{\cite{BVunbdd} }that there exist an open set
$V$, containing some interval of the form $(-a,0)$, and an analytic
function $\Sigma_{\mu}$ defined on $V$ such that 
\[
\eta_{\mu}(z\Sigma_{\mu}(z))=z,\quad z\in V.
\]
The related equation $\eta_{\mu}(z)\Sigma_{\mu}(\eta_{\mu}(z))=z$
holds for $z$ in every connected open set $U\subset\mathbb{C}\setminus\mathbb{R}_{+}$ that intersects ${\mathbb R}_-$ and 
such that $\eta_{\mu}(z)$ belongs to the domain of $\Sigma_{\mu}$.
The function $\Sigma_{\mu}$ serves an analogous role relative to
multiplicative free convolution to that of $\varphi_{\mu}$ relative
to additive free convolution, namely
\[
\Sigma_{\mu\boxtimes\nu}(z)=\Sigma_{\mu}(z)\Sigma_{\nu}(z)
\]
for $z$ in every domain on which the three functions are defined.
The measure $\mu$ is $\boxtimes$-infinitely divisible precisely
when the function $\Sigma_{\mu}$ continues analytically to the entire
domain $\mathbb{C}\backslash\mathbb{R}_{+}$ and this analytic continuation
can be written as
\[
\Sigma_{\mu}(z)=\gamma\exp\left[\int_{[0,+\infty]}\frac{1+tz}{z-t}\,d\sigma(t)\right],\quad z\in\mathbb{C}\backslash\mathbb{R}_{+},
\]
where $\gamma>0$ and $\sigma$ is a finite, positive Borel measure
on the one-point compatification $[0,+\infty]$ of ${\mathbb R}_+$ \cite{BVunbdd}. 

Suppose now that $\mu$ is $\boxtimes$-infinitely divisible. The
equation
\[
\eta_{\mu}(z)\Sigma_{\mu}(\eta_{\mu}(z))=z
\]
extends by analytic continuation to the entire slit plane $\mathbb{C}\backslash\mathbb{R}_{+}.$
In particular, it shows that $\eta_{\mu}$ maps $\mathbb{C}\backslash\mathbb{R}_{+}$
conformally onto a domain $\Omega_{\mu}$ that is symmetric relative
to the real line. In addition to the interval $(-\infty,0]$, the
boundary of $\Omega_{\mu}\cap\mathbb{H}$ consists of  a curve of the
form 
\[
C=\{re^{ih_{\mu}(r)}:r\in(0,+\infty)\},
\]
 where $h_{\mu}:(0,+\infty)\to[0,\pi)$ is a continuous function,
real analytic wherever it is strictly positive. In other words,
\[
\Omega_{\mu}\cap\mathbb{H}=\{re^{i\theta}:r>0,h_{\mu}(r)<\theta<\pi\}.
\]
 Moreover, the function $\eta_{\mu}$ extends continuously and injectively
to $\mathbb{H}\cup(0,+\infty)$ such that the range $\eta_{\mu}\left((0,+\infty)\right)$
is exactly the curve $C$ and that the inversion relationships between
$\eta_{\mu}$ and the map $\Phi(z)=z\Sigma_{\mu}(z)$ now extend to
the boundary of the relevant domains. (See  \textbf{\cite[Section2]{BWZ-mul}}
for a review of these results concerning boundary behavior, and the references therein
for their origin.) Thus, there is at most one value
$t_{0}>0$ such that $\eta_{\mu}(t_{0})=1$. Such a point exists precisely
when $h_{\mu}(1)=0$. Outside possibly the point $1/t_{0}$, the measure
$\mu$ is absolutely continuous and its density is locally analytic
wherever it is strictly positive. We give a short proof of this
analyticity for  future reference. Suppose that $x_{0}>0$ is a point
where the density of $\mu$ is positive. We write $1/x_{0}=r_{0}e^{ih_{\mu}(r_{0})}\Sigma_{\mu}(re^{ih_{\mu}(r_{0})})$
for some $r_{0}>0$ such that $h_{\mu}(r_{0})>0$. The continuity
of $h_{\mu}$ yields $h_{\mu}(r)>0$ for $r$ near $r_{0}$. For such
$r$, the function $\Phi(z)=z\Sigma_{\mu}(z)$ has a non-zero complex
derivative at $z=re^{ih_{\mu}(r)}$. This is because $\Phi^{\prime}(re^{ih_{\mu}(r)})=0$
would imply that the positive number $\Phi(re^{ih_{\mu}(r)})$ has
multiple preimages located on the curve $C$, contradicting  the
fact that $\Phi$ is injective on $C$. Since the inversion equation
\[
\Phi(\eta_{\mu}(z))=z
\]
holds for $z\in\mathbb{H}$ close to $1/x_{0}$, we conclude that
$\eta_{\mu}$ can be analytically continued to a neighborhood of $1/x_{0}$
as the inverse of the function $\Phi$. The desired analyticity now
follows from (\ref{eq:G and eta}) and the Stieltjes inversion formula.
(One can also write an implicit formula for the density, just as
in the additive case, to see its analyticity directly; see for instance
\cite[(3.1)]{BWZ-mul}.)

Suppose now that $k_{1}<k_{2}<\cdots$ is a sequence of natural numbers
and $\{\mu_{n,i}:n,i\in\mathbb{N},i\le k_{n}\}$ is an infinitesimal
array of probability measures on $\mathbb{R}_{+}$; that is,
\[
\lim_{n\to\infty}\min_{1\le i\le k_{n}}\mu_{n,i}((1-\varepsilon,1+\varepsilon))=1
\]
for every $\varepsilon\in(0,1)$. The analog of the domain $\Gamma_{\alpha,\beta}$
in the context of multiplicative free convolution is the domain
\[
\Omega_{\rho,\theta}=\left\{ re^{it}:\rho<r<\frac{1}{\rho},t\in(\theta,2\pi-\theta)\right\} ,
\]
where $\rho>0$ and $\theta\in(0,\pi)$. By the results in \cite{BB-CMS},
given arbitrarily small $\rho>0$ and $\theta\in(0,\pi)$, all the
functions $\Sigma_{\mu_{n,i}}$ are defined in $\Omega_{\rho,\theta}$
if $n$ is sufficiently large. As shown in \cite{BVunbdd}, the weak
convergence of probability measures on $\mathbb{R}_{+}$ can be translated
into local uniform convergence of the corresponding $\eta$-functions,
or into local uniform convergence of the corresponding $\Sigma$-functions
on some $\Omega_{\rho,\theta}$. The following result is analogous
to Theorem \ref{thm:general additive} and therefore it provides a
superconvergence result for infinitesimal arrays.
\begin{thm}
\label{thm:multi R+}Let $\{\nu_{n}\}_{n\in\mathbb{N}}$ be a sequence
of Borel probability measures on $\mathbb{R_{+}}$ that converges
weakly to an $\boxtimes$-infinitely divisible measure $\nu.$ Suppose
that for every $\rho>0$ and $\theta\in(0,\pi)$ there exists $N\in\mathbb{N}$
such that $\Sigma_{\nu_{n}}$ is defined in $\Omega_{\rho,\theta}$
for every $n\ge N$. Let $J\subset(0,+\infty)$ be a compact interval
such that $\nu$ is absolutely continuous and $d\nu/dx>0$ in a neighborhood
of $J$. Then $\nu_{n}$ is absolutely continuous with an analytic
density on a neighborhood of $J$ for $n$ sufficiently large, and
the densities $d\nu_{n}/dx$ converge uniformly on $J$, along with
all their derivatives, to $d\nu/dx$ as $n\to\infty$.
\end{thm}

\begin{proof}
[Sketch of proof]To simplify notation, we set
\[
\Phi(z)=z\Sigma_{\nu}(z)\text{ and }\Phi_{n}(z)=z\Sigma_{\nu_{n}}(z),\quad n\in\mathbb{N}.
\]
Suppose that $x_{0}>0$ is a point where the density of $\nu$ is
positive. As seen above, we have
\[
\frac{1}{x_{0}}=\Phi(r_{0}e^{ih_{\nu}(r_{0})})
\]
for some $r_{0}>0$ such that $h_{\nu}(r_{0})>0$. The point $r_{0}e^{ih_{\nu}(r_{0})}$
belongs to the domain $\Omega_{\rho,\theta}$ provided that $\rho$
and $\theta$ are sufficiently small. Fix such values $\rho$ and
$\theta$, and note that $\Phi_{n}$ is defined on $\Omega_{\rho,\theta}$
provided that $n$ is sufficiently large. As seen in the earlier discussions,
the essential point is to verify that, for $z\in\mathbb{H}$ sufficiently
close to $1/x_{0}$, we have $\eta_{\nu_{n}}(z)\in\Omega_{\rho,\theta}$
and the identity
\begin{equation}
\Phi_{n}(\eta_{\nu_{n}}(z))=z\label{eq:inversion-multi-line}
\end{equation}
 holds for $n$ sufficiently large. The absolute continuity of $\nu_{n}$
near $x_{0}$ would follow as a byproduct. Note that the identity
(\ref{eq:inversion-multi-line}) holds for $z\in(-T,-1/T)$ if $T$
is large enough, and it extends by analytic continuation to the largest
interval of this form with the property that $\eta_{\nu_{n}}((-T,-1/T))\subset\Omega_{\rho,\theta}$. 

We observe now that, by using the polar coordinates, the argument of
Remark \ref{rem:y-derivative again} (or alternatively, more explicit
calculations from \cite[Lemma 4.2]{HZ-MathZ}) yields 
\[
\frac{\partial}{\partial\theta}\arg\left(\Phi(re^{i\theta})\right)>0\text{ for }r=r_{0}\text{ and }\theta=h_{\nu}(r_{0}).
\]
So we may choose a neighborhood $W$ of $r_{0}e^{ih_{\nu}(r_{0})}$
and an integer $N$ such that $\Phi_{n}$ is defined on $\Omega_{\rho,\theta}$
for $n\ge N$, $W\subset\Omega_{\rho,\theta}\cap\mathbb{H}$, and
\[
\frac{\partial}{\partial\theta}\arg\left(\Phi_{n}(re^{i\theta})\right)>0\text{ for \ensuremath{re^{i\theta}\in W\text{ and }n\ge N.}}
\]
In particular, the complex derivative $\Phi_{n}^{\prime}$ does not
vanish on $W$, so that $\Phi_{n}|_{W}$ has an analytic inverse.
Then we choose $\varepsilon>0$ so small that
\begin{enumerate}
\item $\left|h_{\nu}(r)-h_{\nu}(r_{0})\right|<\varepsilon$ for $|r-r_{0}|\le\varepsilon$,
\item the compact set
\[
K=\{re^{i\theta}:|r-r_{0}|\le\varepsilon,|\theta-h_{\nu}(r_{0})|\le\varepsilon\}
\]
is contained in $W$,
\item $\Phi(re^{ih_{\nu}(r_{0})+i\varepsilon})\in\mathbb{H}$ for $|r-r_{0}|\le\varepsilon$,
and
\item $\Phi(re^{ih_{\nu}(r_{0})-i\varepsilon})\in-\mathbb{H}$ for $|r-r_{0}|\le\varepsilon.$
\end{enumerate}
Since $\Phi_{n}$ converges to $\Phi$ uniformly on $K$, we can assume
that properties (2) and (3) also hold for $\Phi_{n}$ after making
$N$ bigger. It follows that there exists a unique function $h_{n}:[r_{0}-\varepsilon,r_{0}+\varepsilon]\to(h_{\nu}(r_{0})-\varepsilon,h_{\nu}(r_{0})+\varepsilon)$
such that
\[
\Phi_{n}(re^{ih_{n}(r)})\in(0,+\infty),\quad|r-r_{0}|\le\varepsilon.
\]
The function $\Phi_{n}$ is one-to-one on $K$, and the interval
\[
\{\Phi_{n}(re^{ih_{n}(r)}):|r-r_{0}|\le\varepsilon\}
\]
is a neighborhood of $1/x_{0}$ for sufficiently large $n$. We need
to show that (\ref{eq:inversion-multi-line}) holds in the set
\[
\{\Phi_{n}(re^{i\theta}):|r-r_{0}|\le\varepsilon,h_{n}(r)<\theta\le h_{\nu}(r)+\varepsilon\},
\]
provided that $n$ is sufficiently large. This will follow by analytic
continuation once we show that $\Phi_{n}(K')\subset\mathbb{H}$, where
\[
K'=\{re^{i\theta}:|r-r_{0}|\le\varepsilon,\theta\in[h_{\nu}(r)+\varepsilon,\pi]\}.
\]
This last fact follows immediately because $\Phi(K')\subset\mathbb{H}$
and $\Phi_{n}$ converges uniformly on $K'$ to $\Phi$.
\end{proof}

The following corollary formalizes the discussion preceding Theorem~\ref{thm:multi R+}.
\begin{cor}
    Suppose that $k_{1},k_{2},\dots\in\mathbb{N}$, and that $\{\mu_{n,i}:n,i\in\mathbb{N},i\le k_{n}\}$
is an infinitesimal array of Borel probability measures on $\mathbb{R}_+$. Set
\[\nu_n = \mu_{n,1}\boxtimes\cdots\boxtimes\nu_{n,k_n}.\]
Assume that $\nu_n$ converges to a $\boxtimes$-infinite divisible measure $\nu$. Then for any compact interval $J\subset (0,\infty)$ such that $\nu$ is absolutely continuous 
and $d\nu/dx>0$ in a neighborhood of $J$, $\nu_n$ is absolutely continuous on a neighborhood of $J$ with a real-analytic density for sufficiently large $n$, 
and the densities $d\nu_n/dx$ converge uniformly on $J$ to $d\nu/dx$ as $n\to\infty$.
\end{cor}

\section{Multiplicative convolution on $\mathbb{T}$\label{sec:Multiplicative-convolution-onT}}

Finally, we consider the superconvergence phenomenon on the unit circle
$\mathbb{T}=\{e^{i\theta}:\theta\in[0,2\pi)\}$. If $\mu$ is a Borel
probability measure on $\mathbb{T}$, one sets again
\[
\psi_{\mu}(z)=\int_{\mathbb{T}}\frac{z\zeta}{1-z\zeta}\,d\mu(\zeta),\quad\eta_{\mu}(z)=\frac{\psi_{\mu}(z)}{1+\psi_{\mu}(z)},
\]
but these functions are now defined on the unit disk $\mathbb{D}=\{z\in\mathbb{C}:|z|<1\}$.
The density of $\mu$ relative to the normalized arclength measure
$m=d\theta/2\pi$ is given almost everywhere by 
\[
\frac{d\mu}{dm}(\xi)=\lim_{r\uparrow1}\Re\frac{1+\eta_{\mu}(r\,\overline{\xi})}{1-\eta_{\mu}(r\,\overline{\xi})},\quad\xi\in\mathbb{T}.
\]
We restrict our considerations to the case in which $\int_{\mathbb{T}}\zeta\,d\mu(\zeta)\ne0$
and we observe that this condition is satisfied for (all but finitely
many) measures in an infinitesimal array. In terms of $\psi$ and
$\eta$, this condition amounts to the requirement that $\psi_{\mu}'(0)=\eta_{\mu}'(0)\ne0$.
This implies the existence of an analytic function $\Sigma_{\mu}$,
defined in a disk $\rho\mathbb{D}=\left\{ z\in\mathbb{C}:\left|z\right|<\rho\right\} $
with $\rho\in(0,1)$ and with values in $\mathbb{D},$ such that
\[
\eta_{\mu}(z\Sigma_{\mu}(z))=z,\quad z\in\rho\mathbb{D}.
\]
We denote by $\rho_{\mu}$ the radius of convergence of the Taylor
series of $\Sigma_{\mu}$, so $\Sigma_{\mu}$ is defined in $\rho_{\mu}\mathbb{D}$.
The free multiplicative convolution of measures on $\mathbb{T}$ satisfies
the identity
\[
\Sigma_{\mu\boxtimes\nu}(z)=\Sigma_{\mu}(z)\Sigma_{\nu}(z),\quad|z|<\min\{\rho_{\mu},\rho_{\nu}\}.
\]

Among the $\boxtimes$-infinitely divisible measures on $\mathbb{T}$,
the normalized arclength measure $m$ is the only one with zero first moment. The other $\boxtimes$-infinitely
divisible measures $\mu$ on $\mathbb{T}$ are characterized by the
fact that $\rho_{\mu}\ge1$ and $\Phi_{\mu}(\mathbb{D})\subset\mathbb{D},$
where $\Phi_{\mu}(z)=z\Sigma_{\mu}(z)$. Clearly, if $\mu$ is $\boxtimes$-infinitely
divisible, the equation
\begin{equation}
\Phi_{\mu}(\eta_{\mu}(z))=z\label{eq:inversionT}
\end{equation}
extends by analytic continuation to the entire disk $\mathbb{D}$,
showing that $\eta_{\mu}$ maps $\mathbb{D}$ conformally onto a domain
$\Omega_{\mu}\subset\mathbb{D}.$ In fact, by the results in \cite{HZ-MathZ} (see
also  \cite[Section 5]{BWZ-mul} for a review), $\eta_{\mu}$ extends
to a homeomorphism of $\overline{\mathbb{D}}$ onto $\overline{\Omega_{\mu}}$.
In particular, $\eta_{\mu}$ maps $\mathbb{T}$ homeomorphically onto
$\partial\Omega_{\mu}$. The domain $\Omega_{\mu}$ is starlike, in
particular its boundary is a closed curve of the form
\[
\{R_{\mu}(\zeta)\zeta:\zeta\in\mathbb{T}\},
\]
where $R_{\mu}:\mathbb{T}\to(0,1]$ is a continuous function that
is real-analytic at all points $\zeta$ satisfying $R_{\mu}(\zeta)<1$.
Thus, one has 
\[
\Omega_{\mu}=\{r\zeta:\zeta\in\mathbb{T},0\le r<R_{\mu}(\zeta)\}.
\]
 For a $\boxtimes$-infinitely divisible measure $\mu$, there is
at most one point $\zeta_{0}$ such that $\eta_{\mu}(\zeta_{0})=1$,
and this happens precisely when $R_{\mu}(1)=1$. Outside the point
$\overline{\zeta_{0}}$, the measure $\mu$ is absolutely continuous
with a locally analytic density wherever this density is positive.
This is a consequence of the identity (\ref{eq:inversionT}) that
holds for $z\in\mathbb{D}$ close to a point $\xi_{0}\in\mathbb{T}$
such that $|\eta_{\mu}(\xi_{0})|<1$ because (as we also discuss in
the proof below) $\Phi'(\eta_{\mu}(\xi_{0}))\ne0$. 

Suppose that $k_{1}<k_{2}<\cdots$ is a sequence of natural numbers
and $\{\mu_{n,i}:n,i\in\mathbb{N},i\le k_{n}\}$ is an infinitesimal
array of Borel probability measures on $\mathbb{T}$, that is,
\[
\lim_{n\to\infty}\min_{1\le i\le k_{n}}\mu_{n,i}(\{\zeta\in\mathbb{T}:|\zeta-1|<\varepsilon\})=1
\]
for every $\varepsilon>0$. According to \cite{BB-CMS}, given an
arbitrary $\rho\in(0,1)$, we have $\rho_{\mu_{n,i}}>\rho$ for all
$1\leq i\leq k_{n}$ provided that $n$ is sufficiently large.

The weak convergence of probability measures on $\mathbb{T}$ can
be translated into local uniform convergence of the $\eta$-functions. If all the measures under consideration satisfy $\int_{\mathbb{T}}\zeta\,d\mu(\zeta)\ne0$
it also translates into uniform convergence of $\Phi$-functions on
$\rho\mathbb{D}$ for some $\rho>0$.

We are now ready to state the circle version of the superconvergence
result.
\begin{thm}
    \label{thm:multi T}
Let $\{\nu_{n}\}_{n\in\mathbb{N}}$ be a sequence of Borel probability
measures on $\mathbb{T}$ that converges weakly to an $\boxtimes$-infinitely
divisible measure $\nu$, where $\int_{\mathbb{T}}\zeta\,d\nu(\zeta)\ne0$.
Suppose that $\lim_{n\to\infty}\rho_{\nu_{n}}=1$. Let $J\subset\mathbb{T}$
be a compact arc such that $\nu$ is absolutely continuous and its
density is positive in a neighborhood of $J$. Then $\nu_{n}$ is
absolutely continuous with a real-analytic density on a neighborhood
of $J$ for $n$ sufficiently large, and the densities of $\nu_{n}$
converge uniformly on $J$, along with all their derivatives, to that
of $\nu$ as $n\to\infty$.
\end{thm}

\begin{proof}
[Sketch of proof]We use the notation $\Phi=\Phi_{\nu}$ and $\Phi_{n}=\Phi_{\nu_{n}}$,
$n\in\mathbb{N}.$ Suppose that $\xi_{0}\in\mathbb{T}$ is such that
$\eta_{\nu}(\xi_{0})\in\mathbb{D}$. As in the proof of Theorem \ref{thm:multi R+},
we need to show that, for sufficiently large $n$, we have $\eta_{\nu_{n}}(z)\in\rho_{\nu_{n}}\mathbb{D}$
for $z\in\mathbb{D}$ close to $\xi_{0}$, and that $\Phi_{n}(\eta_{\nu_{n}}(z))=z$
for such $z$. By hypothesis, we may assume that $\rho_{\nu_{n}}>\rho>|\eta_{\nu}(\xi_{0})|$
for some $\rho\in(0,1),$ and that $|\Phi_{n}(z)|<1$ for $|z|<\rho$.
It is easily seen, by using logrithmic polar coordinate, from Remark
\ref{rem:y-derivative again} (see also  \cite[Section 3]{HZ-MathZ}
for a proof by explcit formulas) that
\[
\frac{\partial}{\partial r}\log\left|\Phi(re^{i\theta})\right|>0
\]
at the point $r_{0}e^{i\theta_{0}}=\eta_{\nu}(\xi_{0})$. Thus, we
may also assume that 
\[
\frac{\partial}{\partial r}\log\left|\Phi_{n}(re^{i\theta})\right|>0
\]
for $re^{i\theta}$ in a fixed neighborhood of $\eta_{\nu}(\xi_{0})$
and for large $n$. It follows that there is a sequence $\{R_{n}\}_{n\in\mathbb{N}}$
of functions defined in a neighborhood of $e^{i\theta_{0}}$, with
values in $(0,1)$, such that 
\[
|\Phi_{n}(r\zeta)|\begin{cases}
<1, & r<R_{n}(\zeta),\\
=1, & r=R_{n}(\zeta),\\
>1, & r>R_{n}(\zeta).
\end{cases}
\]
Moreover, the functions $R_{n}$ converge uniformly to $R_{\nu}$
in a neighborhood of $e^{i\theta_{0}}$. Then, we need to verify that
\[
\eta_{\nu_{n}}(\Phi_{n}(z))=z
\]
for $z=r\zeta$ provided that $r<R_{n}(\zeta)$ and $\zeta$ is close
to $e^{i\theta_{0}}$, say $|\zeta-e^{i\theta_{0}}|\le\varepsilon$.
This statement follows by analytic continuation because $\Phi_{n}$
converges uniformly to $\Phi$ in a compact set of the form
\[
\{r\zeta:|\zeta-e^{i\theta_{0}}|\le\varepsilon,r\le R_{\nu}(e^{i\theta_{0}})-\delta\},
\]
where $\varepsilon,\delta>0$. Absolute continuity of $\nu_{n}$ near
$1/\xi_{0}$, and the uniform convergence of the densities,  follow
as in the free additive case. 
\end{proof}

As in the additive case and the multiplicative case on $\mathbb{R}_+$, we have the following corollary that formalizes the discussion preceding Theorem~\ref{thm:multi T}.
\begin{cor}
    Suppose that $k_{1},k_{2},\dots\in\mathbb{N}$, and that $\{\mu_{n,i}:n,i\in\mathbb{N},i\le k_{n}\}$
is an infinitesimal array of Borel probability measures on $\mathbb{T}$. Set
\[\nu_n = \mu_{n,1}\boxtimes\cdots\boxtimes\nu_{n,k_n}.\]
Assume that $\nu_n$ converges to a $\boxtimes$-infinite divisible measure $\nu$ such that $\nu\ne m$. Then for any compact arc $J\subset \mathbb{T}$ such that $\nu$ is absolutely continuous 
and $d\nu/dx>0$ in a neighborhood of $J$, $\nu_n$ is absolutely continuous on a neighborhood of $J$ with a real-analytic density for sufficiently large $n$, 
and the densities $d\nu_n/dx$ converge uniformly on $J$ to $d\nu/dx$ as $n\to\infty$.
\end{cor}

\end{document}